\documentclass[11pt,oneside]{amsart}

\usepackage{amsfonts}
\usepackage{amsmath, amsthm, amssymb, ulem, amscd} 
\usepackage[mathscr]{euscript}

\def\crn#1#2{{\vcenter{\vbox{
        \hbox{\kern#2pt \vrule width.#2pt height#1pt
           }
          \hrule height.#2pt}}}}

\newcommand{\pa}{\partial}

\newcommand{\Vol}{\operatorname{Vol}}

\newcommand{\tr}{\operatorname{tr}}

\newcommand{\ep}{\epsilon}
\newcommand{\fe}{\varphi}
\newcommand{\al}{\alpha}

\newcommand{\ga}{\gamma}
\newcommand{\de}{\delta}
\newcommand{\be}{\beta}

\newcommand{\om}{\omega}
\newcommand{\Up}{\Upsilon}

\newcommand{\gbh}{\widehat{\overline{g}}}
\newcommand{\gb}{\overline{g}}
\newcommand{\nb}{\overline{n}}
\newcommand{\nab}{\overline{\nabla}}

\newcommand{\Rb}{\overline{R}}
\newcommand{\Pb}{\overline{P}}

\newcommand{\Vh}{\widehat{V}}

\newcommand{\eh}{\widehat{\epsilon}}
\newcommand{\rh}{\widehat{r}}
\newcommand{\uh}{\widehat{u}}

\newcommand{\Si}{\Sigma} 

\newcommand{\cL}{\mathcal{L}}

\newcommand{\cE}{\mathcal{E}}

\theoremstyle{plain}
\newtheorem{theorem}{Theorem}[section]

\newtheorem{proposition}[theorem]{Proposition}

\theoremstyle{definition}

\theoremstyle{remark}
\newtheorem{remark}[theorem]{Remark}

\numberwithin{equation}{section}

\title[Volume Renormalization for Singular Yamabe Metrics]{Volume 
Renormalization for Singular Yamabe Metrics}

\author{C. Robin Graham}
\address{Department of Mathematics, University of Washington,
Box 354350\\
Seattle, WA 98195-4350, USA}
\email{robin@math.washington.edu}

\begin{document}

\maketitle

\thispagestyle{empty}

\renewcommand{\thefootnote}{}
\footnotetext{Research partially supported by NSF grant \# DMS
  1308266}     
\renewcommand{\thefootnote}{1}

\section{Introduction}\label{introduction}

In \cite{LN}, Loewner and Nirenberg introduced what is now known as the 
singular 
Yamabe problem.  The case of interest here can be formulated as follows.
Given a smooth compact Riemannian manifold-with-boundary $(M^{n+1},\gb)$,
find a defining function $u$ for $\partial M = \Sigma$ so that the scalar  
curvature of the metric $g=u^{-2}\gb$ on $\mathring M$ satisfies 
$R_g=-n(n+1)$.  This 
problem has an obvious conformal invariance:  the rescaling 
$\widehat{\gb} = \Omega^2 \gb$, where $0<\Omega \in C^\infty(M)$,
induces the rescaling $\uh = \Omega u$ leaving $g$ unchanged.  Hence the
datum for the problem can be considered as a conformal class of metrics on
$M$, and its solution produces from the datum a canonical conformally  
compact metric in the conformal class on $\mathring M$.  

It follows from results in \cite{LN}, \cite{AM}, \cite{ACF} that for $n\geq
2$ there always exists a unique solution $u$.  However, $u$ 
might not be smooth up to $\partial M$.  In \cite{ACF},
Andersson, Chru\'sciel and Friedrich showed that $u$ has an   
asymptotic expansion involving powers of $r$ and $\log r$, where $r$ is a
smooth defining function for $\Sigma$.  Moreover, the expansion to order
$n+2$ is locally determined, and there is a single conformally invariant, 
locally determined density on $\Sigma$, the coefficient of the first  
log term $r^{n+2}\log r$ in the expansion, which, if nonzero, 
obstructs smoothness of $u$. 

In \cite{GoW2}, \cite{GoW3}, \cite{GGHW}, Gover and Waldron and
collaborators have developed the singular Yamabe problem as a tool for
studying the geometry of a hypersurface in a conformal manifold, in the
process applying and further developing their boundary calculus for
conformally compact manifolds (\cite{GoW1}).  A starting point for
these investigations was the observation that when $n=2$ and $\gb$ is
Euclidean, the obstruction identified in \cite{ACF} is 
the Willmore invariant of $\Sigma$, i.e. the variational
derivative of the Willmore energy with respect to variations of $\Sigma$.
This led them to interpret the  
obstruction in general as a higher-dimensional generalization of the
Willmore invariant.  In \cite{GoW2}, they raised the question of  
whether in higher dimensions, there is a conformally invariant energy,  
generalizing the Willmore energy, whose variational derivative is the
singular Yamabe obstruction.  The main purpose of this paper is to show 
that such an energy can be constructed by renormalizing the volume of the
singular Yamabe metric.  After we described this work to Gover and
Waldron and showed them our proof, they posted \cite{GoW4}, which also  
discusses volume renormalization for singular Yamabe metrics and proves
the same result.  

The volume renormalization process carried out here is the generalization
to singular Yamabe metrics of volume renormalization for
Poincar\'e-Einstein metrics, introduced in the physics literature and
described in \cite{G1}. 
We consider the asymptotics as $\ep\rightarrow 0$ of 
$\Vol_g(\{r>\ep\})$, where $r$ is the distance to $\Sigma$ in the metric
$\gb$.  The volume form of $g$ has a pointwise expansion \eqref{volform},
\eqref{expand} whose coefficients
$v^{(k)}$ are invariants of the hypersurface $\Sigma$ in the Riemannian
manifold $(M,\gb)$.  These coefficients generalize to the singular Yamabe
problem the 
Poincar\'e-Einstein renormalized volume coefficients which have been the
subject of recent investigations (\cite{GJ}, \cite{CF}, \cite{J1}, 
\cite{G2}, \cite{CFG}, \cite{J2}).  Upon integration, the pointwise
expansion produces an expansion \eqref{sypvolexp} for    
$\Vol_g(\{r>\ep\})$ whose coefficients \eqref{coeffs} are constant
multiples of the integrals of the $v^{(k)}$.  
Our energy is the coefficient of the $\log\frac{1}{\ep}$ term in the
renormalized volume expansion.  It is the integral of $v^{(n)}$ and is
invariant under conformal rescalings of $\gb$.  In \cite{GoW4}, a
version of $Q$-curvature is defined for the singular Yamabe setting and the 
energy is shown to equal the integral of the $Q$-curvature.   
The fact that the variation of the energy with respect to $\Sigma$ is the
log term coefficient in the 
solution $u$ is an analog of the result of \cite{HSS}, \cite{GH} that  
in the Poincar\'e-Einstein case with even-dimensional boundary, the metric 
variation of the log term coefficient in the volume expansion is a constant
multiple of the ambient obstruction tensor, the log term coefficient in the
expansion of the Poincar\'e-Einstein metric itself.  

There is a similar renormalization for 
the area of minimal submanifolds of Poincar\'e-Einstein spaces, also
introduced and used extensively in the physics literature.  Forthcoming
work with Nicholas Reichert will analyze in some detail the corresponding
energy in this setting, introduced in \cite{GrW}, and, among other things,    
will prove the analogous result:  for even dimensional submanifolds, the 
submanifold variation of the log term coefficient in the renormalized area
expansion agrees with the obstruction to smoothness for minimal extension.       
Unlike the extension problems for Poincar\'e-Einstein metrics and minimal
submanifolds, in which log terms occur only for
even-dimensional boundaries, smoothness for the singular Yamabe  
problem is generically obstructed in all dimensions $n>1$ (\cite{GoW3}).     
Another distinction is that in the singular Yamabe problem, the conformal
rescaling happens in the same space in which the extension problem is
posed, rather than on the boundary at infinity.  

{\it Note about terminology.}  In discussing renormalized volume 
generally, we use {\it energy} for the coefficient of 
the log term in the volume expansion, {\it renormalized
volume} for the constant term, and {\it anomaly} for 
the difference of the renormalized volumes corresponding to different
choices of conformal representatives.  The energy is conformally invariant
and is the integral of a local scalar invariant of the geometry
determined by fixing a metric in the conformal class.   
The renormalized volume is global and in general is not
conformally invariant.  The anomaly is the integral of a locally determined  
nonlinear differential operator which depends on local background geometry
applied to the conformal factor.  In the case of a constant conformal  
factor, the anomaly reduces to a multiple of the energy.  But in general, 
the anomaly contains more information.  For pure conformal geometry, the
linearized anomaly determines a  
particular integrand (namely $v^{(n)}$) for the energy, so fixes divergence
terms.  Anomalies associated to submanifolds contain still 
more information:  since the rescaling occurs in the full space but
the integration is over the submanifold, normal derivatives of the
conformal factor appear in the anomaly as well.

In \S\ref{asymptotics} we review the asymptotics of solutions of the
singular Yamabe problem, discuss the renormalized volume expansion, and
describe how Poincar\'e-Einstein volume renormalization is a special case.
In \S\ref{variations} we formulate and prove that the variation of the
energy is the singular Yamabe obstruction.  
In \S\ref{calc} we carry out the calculations outlined in
\S\ref{asymptotics} far enough to identify the first two renormalized
volume coefficients $v^{(1)}$, $v^{(2)}$ for general $n$, and the energy
$\cE$ and the anomaly for $n=2$.  We also compare the $n=2$ energy and
anomaly with the corresponding quantities derived in \cite{GrW} for the
renormalization of the area of a minimal submanifold of the corresponding
Poincar\'e-Einstein space with boundary at infinity equal to $\Sigma$.

\section{Asymptotics}\label{asymptotics}

Let $(M^{n+1},\gb)$, $n\geq 1$ be a Riemannian manifold-with-boundary and
denote $\pa M = \Si$.  We search for a defining  
function $u$ of $\Si$ so that $g=u^{-2}\gb$ has constant scalar curvature
$R_g=-n(n+1)$.  Recall the conformal change of scalar curvature in the
form:   
\begin{equation}\label{confR}
R_{u^{-2}\gb}=-n(n+1)|du|^2_{\gb} +2nu\Delta_{\gb} u +u^2R_{\gb},
\end{equation}
where $\Delta_{\gb}=\gb^{\al\be}\nab_{\al}\nab_{\be}$.
So the singular Yamabe problem amounts to solving the equation
\begin{equation}\label{singyam}
n(n+1)=n(n+1)|du|_{\gb}^2-2nu\Delta_{\gb} u -u^2R_{\gb}. 
\end{equation}

The normal exponential map $\exp:[0,\de)_r\times \Sigma\rightarrow M$ 
relative to $\gb$ is a diffeomorphism onto a neighborhood of
  $\Sigma$, with respect to which $\gb$ takes the form 
\begin{equation}\label{geoform}
\gb=dr^2+h_r
\end{equation}
for a one-parameter family of metrics $h_r$ on $\Sigma$.  So $r$ is the 
$\gb$-distance to $\Sigma$ and $h_0$ the induced metric.  
We use $\al$, $\be$ as indices for objects on $M$, $i$, $j$ for objects
on $\Sigma$, and $0$ for the $r$ factor.  Thus $\al$ corresponds to the
pair $(i,0)$ relative to the product identification induced by $\exp$.  
The derivatives 
$\pa_r^k h_r$ at $r=0$ can be expressed in terms of the
curvature of $\gb$, its covariant derivatives, and the second fundamental
form, which we denote $L_{ij}$.  For instance, denoting $\partial_r$ by
$'$, one has at $r=0$
\begin{equation}\label{derivs}
h'_{ij}=-2L_{ij},\qquad h''_{ij}=-2\Rb_{0i0j} +2L_{ik}L_j{}^k.
\end{equation}
For $\gb$ of the form \eqref{geoform}, we have 
$\Delta_{\gb}u=\pa_r^2u+\frac12 h^{ij}h'_{ij}\pa_ru+\Delta_{h_r}u$.
Thus \eqref{singyam} becomes 
$$
n(n+1)=n(n+1)\left((\pa_ru)^2+h^{ij}\pa_iu\pa_ju\right)
-2nu\left(\pa_r^2u+\tfrac12
h^{ij}h'_{ij}\pa_ru+\Delta_{h_r}u\right)-u^2R_{\gb}.   
$$  

Consider the formal asymptotics of $u$.  Setting $r=0$ and recalling
that $u=0$ when $r=0$ and $u>0$ for $r>0$, one 
concludes that $\pa_ru=1$ at $r=0$.  So write 
$u=r+r^2\fe$.  In terms of $\fe$, the equation becomes   
\begin{equation}\label{fe}
\begin{split}
(1+r\fe)&\Big[ r^2\fe_{rr}+4r\fe_r+2\fe 
+\tfrac12 h^{ij}h'_{ij} (1+2r\fe+r^2\fe_r)+r^2\Delta_{h_r}\fe\Big]\\
-\frac{n+1}{2}&\Big[2(r\fe_r+2\fe)+r(r\fe_r+2\fe)^2+r^3h^{ij}\pa_i\fe\pa_j\fe\Big]\\
+\frac{1}{2n}&r(1+r\fe)^2R_{\gb}=0.
\end{split}
\end{equation}
The Taylor expansion of $\fe$ can be derived by successive differentiation
of this equation at $r=0$.  Just setting $r=0$ gives
\begin{equation}\label{fe0}
\fe|_{r=0} = \frac{1}{4n}h^{ij}h_{ij}' = -\frac{1}{2n}H,
\end{equation}
where $H=h^{ij}L_{ij}$ is the mean curvature.  
Applying $\pa_r^k$ at $r=0$ gives
$$
(k-n)(k+2)\pa_r^k\fe|_{r=0} = \text{lots},
$$
where lots denotes an expression in lower order derivatives of $\fe$ which
have already been determined.  So $\pa_r^k\fe|_{r=0}$ is formally
determined for $1\leq k\leq n-1$, and there is a potential obstruction in
solving for $\pa_r^n\fe|_{r=0}$ which can be resolved by including a term
in the expansion of 
$\fe$ of the form $r^n\log r$.  It follows that we can uniquely determine
functions $\cL$ and $u^{(k)}$, $2\leq k\leq n+1$, on $\Sigma$, so that if
we set 
\begin{equation}\label{uexpand}
u=r+u^{(2)}r^2+\ldots +u^{(n+1)}r^{n+1}+\cL r^{n+2}\log r,
\end{equation}
then $g=u^{-2}\gb$ satisfies  
\begin{equation}\label{almostconst}
R_g = -n(n+1)+O(r^{n+2}\log r).
\end{equation}
The log term coefficient $\cL$ is the singular Yamabe obstruction, a scalar
field on $\Si$.  One sees easily that under a conformal change
$\gbh = \Omega^2 \gb$, $\cL$ transforms by
$\widehat{\cL}=\big(\Omega|_\Sigma\big)^{-n-1}\cL$.

The function $u$ given by \eqref{uexpand} has been defined near $\Sigma$ in
terms of the product identification determined by the exponential map.
Below we consider the asymptotics of the 
global quantity $\operatorname{Vol}_g(\left\{r>\ep\right\})$.  As our
primary interest is in locally determined quantities near $\Sigma$, in such
global considerations it will suffice 
to take $u$ to be any positive function on $M$ with an asymptotic expansion
which agrees $\mod O(r^{n+2})$ with \eqref{uexpand}.  According to
\cite{ACF}, the exact solution $u$ for which   
$R_g=-n(n+1)$ has this property. 

The volume form of $g=u^{-2}\gb$ is given by 
\begin{equation}\label{volform}
dv_g=u^{-n-1}dv_{\gb}
=r^{-n-1}(1+r\fe)^{-n-1}dv_{h_r}dr
=r^{-n-1}(1+r\fe)^{-n-1}\sqrt{\frac{\det h_r}{\det h_0}}dv_{h_0}dr.
\end{equation}
We can expand
\begin{equation}\label{expand}
(1+r\fe)^{-n-1}\sqrt{\frac{\det h_r}{\det h_0}}
=1+v^{(1)}r+v^{(2)}r^2+\cdots +v^{(n)}r^{n}+O(r^{n+1}\log r).
\end{equation}
The coefficients $v^{(k)}\in C^\infty(\Sigma)$ are the singular Yamabe
renormalized volume coefficients.  It follows that
\begin{equation}\label{sypvolexp}
\operatorname{Vol}_g(\left\{r>\ep\right\}) =\int_{r>\ep}dv_g 
= c_0\ep^{-n} +c_1\ep^{-n+1}+\cdots +c_{n-1}\ep^{-1}
+\cE\log \frac{1}{\ep}+V +o(1) 
\end{equation}
with 
\begin{equation}\label{coeffs}
c_k=\frac{1}{n-k}\int_\Si v^{(k)}\,dv_{h_0}, 
\quad 0\leq k\leq n-1,\qquad \cE=\int_\Si v^{(n)}\,dv_{h_0}.
\end{equation}
$\cE$ is the singular Yamabe energy of $\Si$ and $V$ is the 
renormalized volume of $(M,g)$ with respect to the representative metric
$\gb$.  We will see by direct calculation in \S\ref{calc} that if $n=1$,
then $\cL=0$ and $v^{(1)}=0$, and so also $\cE=\int_\Si v^{(1)}\,dv_{h_0}=0$. 

\begin{proposition}\label{syeinv}
$\cE$ is invariant under conformal changes of $\gb$.  
\end{proposition}
\begin{proof}
Let $\gbh=e^{2\om}\gb$ be a conformally related metric, and $\rh$ the
distance to $\Sigma$ with respect to $\gbh$.  Then $\rh=e^{\Up}r$ for some
smooth function $\Up$.  We need to show that the log term coefficients
$\cE$ in 
the volume expansions \eqref{sypvolexp} for $\gb$ and $\gbh$ agree.  We
will derive an expression for the difference of the volume
expansions from which this is immediate, and which we will use in
\S\ref{calc} to calculate the anomaly.  

Use the normal exponential map of $\gb$ to identify $M$ near $\Sigma$
with $[0,\de)_r\times \Sigma$ as above, and denote points of 
$\Sigma$ by $x$.  For fixed $x$, we can solve the relation
$\rh=e^{\Up(x,r)}r$ for $r$ as a function of $\rh$:  $r=\rh b(x,\rh)$,
where $b(x,\rh)$ is a smooth nonvanishing function.  Set $\eh(x,\ep) = \ep 
b(x,\ep)$.  Then $\rh>\ep$ is equivalent to $r>\eh(x,\ep)$. Recalling 
\eqref{volform}, \eqref{expand}, we have 
\begin{equation}\label{diff}
\begin{split}
\mbox{Vol}_g&(\{r>\ep\})-\mbox{Vol}_g(\{\rh>\ep\})
=
\int_{r>\ep}dv_g - \int_{\rh>\ep}dv_g \\
=&\int_\Sigma\int_{\ep}^{\eh}
\sum_{0\leq k\leq n} v^{(k)}(x) r^{-n-1+k}dr dv_{h_0} + o(1)\\
=&
\sum_{0\leq k\leq n-1}\ep^{-n+k}
\int_\Sigma\frac{v^{(k)}(x)}{-n+k}\left ( b(x,\ep)^{-n+k}-1
\right)dv_{h_0}  + \int_\Sigma v^{(n)}(x)\log b(x,\ep)\,dv_{h_0} 
+o(1). 
\end{split}
\end{equation}
Clearly this expression has no $\log {\frac{1}{\ep}}$ term as  
$\epsilon \rightarrow 0$.  
\end{proof}

The anomaly $V-\Vh$ measures the failure of conformal 
invariance of $V$ under the rescaling $\gbh=e^{2\om}\gb$.  Clearly $V-\Vh$
is the constant 
term in the expansion in $\ep$ of the last line of \eqref{diff}.
It follows from this characterization using the kind of analysis that we 
use in \S\ref{calc} that $V-\Vh$ can be expressed as
the integral over $\Sigma$ of a polynomial expression in $\om$ and its
derivatives whose coefficients depend on derivatives of 
curvature and second fundamental form for $\gb$.    

A Poincar\'e-Einstein metric $g$ has constant scalar curvature, so is the 
singular Yamabe metric in its conformal class.  If $g$ is written near
$\Sigma$ in asymptotically hyperbolic normal 
form $g=r^{-2}(dr^2+h_r)$, and we choose $\gb=dr^2+h_r$ near $\Sigma$, then
the associated singular Yamabe defining   
function $u$ is equal to $r\mod O(r^{n+2})$, and the volume renormalization
expansion \eqref{sypvolexp}, \eqref{coeffs} reduces to the usual 
Poincar\'e-Einstein volume renormalization.  For
Poincar\'e-Einstein metrics, the Taylor expansion of $h_r$ to order $n$ is
determined by $h_0$ and is even in $r$, and the singular Yamabe
renormalized volume coefficients $v^{(k)}$ and energy $\cE$ reduce to the
corresponding Poincar\'e-Einstein coefficients for $h_0$.  If  
$n$ is odd, the energy and anomaly vanish.  The obstruction $\cL$ vanishes
in all dimensions since $u=r \mod O(r^{n+2})$ is smooth (the global term in
the expansion of $h_r$ at order $n$ and the log term for $n$ even arising
from the obstruction tensor do not affect $\cL$ since  
they are trace-free).  We note, however, that even for Poincar\'e-Einstein 
metrics, 
the identification of the variation of $\cE$ in Theorem~\ref{sypvariation}
is a different result from that of \cite{HSS}, \cite{GH}.  The functional
$\cE$ is the same in both cases, but Theorem~\ref{sypvariation} varies
$\Sigma$ with the conformal class of $g$ fixed, while \cite{HSS}, \cite{GH}
varies the conformal infinity $h_0$.

\section{Variations}\label{variations}  

Let again $(M,\gb)$ be our Riemannian manifold, and now let $F_t:\Sigma 
\rightarrow M$, $0\leq t< \delta$ be a variation 
of $\Sigma$, i.e. a smoothly varying one-parameter family of embeddings
with $F_0=\text{Id}$.  Set $\Sigma_t=F_t(\Sigma)$ and $\cE_t = 
\int_{\Sigma_t}v^{(n)}_tdv_{\Sigma_t}$.  Also set $\dot {F} =\pa_t F|_{t=0}
\in \Gamma(TM|_\Sigma)$ and $\dot{\cE}=\pa_t \cE_t|_{t=0}$.  Let $\nb$ 
denote the inward pointing $\gb$-unit normal to $\Si$ in $M$.  

\begin{theorem}\label{sypvariation}
If $n\geq 1$, then 
$$
\dot{\cE} = (n+2)(n-1) \int_\Sigma \langle\dot{F},\nb\rangle_{\gb}\, \cL 
\,dv_\Sigma .   
$$
\end{theorem}

\begin{remark}
As noted above and derived in \S\ref{calc}, both $\cL$ and $\cE$ vanish in
case $n=1$.   
\end{remark}

As in \cite{GH}, the main step in the proof is to express the variation 
of $\operatorname{Vol}_g(\left\{r>\ep\right\})$ as a boundary integral, in 
which the log term in the expansion of $g$ will appear.  In  
both cases, the metrics have constant scalar curvature, so 
$-n(n+1)\operatorname{Vol}_g(\left\{r>\ep\right\})=\int_{r>\ep}R_g\,dv_g$.
In the Poincar\'e-Einstein case, this is the Einstein-Hilbert 
action, which is critical for Einstein metrics.  In the singular Yamabe
case, the variations are within a conformal class, and the total scalar
curvature is critical for constant scalar curvature metrics.  In both
cases, this criticality is used to write the variation as a boundary
integral.

\begin{proof}
Again use the normal exponential map of $\gb$ to identify $M$ near $\Sigma$ 
with $\Sigma \times [0,\de)_r$.  Now $\dot{\cE}$ is a linear functional of  
$\dot{F}$ which depends only 
on the normal component $\langle\dot{F},\nb\rangle_{\gb}$, since $\cE$ is
independent of reparametrizations of $\Sigma$.  Thus it suffices to take 
$F_t(x)=(x,tf(x))$ so that    
$\Si_t=\{r=tf(x)\}$ for some function $f$ on $\Si$.  
Then $\langle\dot{F},\nb\rangle_{\gb}=f$.  Let $r_t$ denote the
geodesic distance to $\Si_t$.  Then the approximate singular Yamabe
defining function $u_t$ for $\Si_t$ analogous to \eqref{uexpand} takes the
form  
\begin{equation}\label{utexpand}
u_t=r_t+u_t^{(2)}r_t^2+\ldots +u_t^{(n+1)}r_t^{n+1}+\cL_t r_t^{n+2}\log r_t
\end{equation}
relative to the product decomposition of $M$ determined by the exponential
map of $\Si_t$.  This $u_t$ has the property that $g_t=u_t^{-2}\gb$
satisfies $R_{g_t}=-n(n+1) +O(r_t^{n+2}\log r_t)$.  Differentiating the
volume expansion \eqref{sypvolexp} gives 
$$
\operatorname{Vol}_{g_t}(\left\{r_t>\ep\right\})\,\dot{}
= \dot{c}_0\ep^{-n} +\dot{c}_1\ep^{-n+1}+\ldots +\dot{c}_{n-1}\ep^{-1}
+\dot{\cE}\log \frac{1}{\ep}+O(1). 
$$
So $\dot{\cE}$ is the coefficient of $\log\frac{1}{\ep}$ in the expansion
of $\operatorname{Vol}_{g_t}(\left\{r_t>\ep\right\})\,\dot{}$.    

Now 
\begin{equation}\label{sum}
\begin{split}
\operatorname{Vol}_{g_t}(\{r_t>\ep\})\,\dot{}
&=\Big(\int_{r_t>\ep}dv_{g_t}\Big)\,\dot{}\\
&=\Big(\int_{r_t>\ep}dv_g\Big)\,\dot{}+\int_{r>\ep}\big(u_t^{-(n+1)}\big)\,\dot{}\,dv_{\gb}\\
&=\Big(\int_{r_t>\ep}dv_g\Big)\,\dot{}-(n+1)\int_{r>\ep}\dot{u}u^{-1}\,dv_{g}.
\end{split}
\end{equation}
We identify the coefficient of $\log \frac{1}{\ep}$ in the expansion of
each of the two terms on the last line of \eqref{sum}. 

For the first term, observe that $\{r_t>\ep\}$ can alternately be written
as $\{r>\psi(x,t,\ep)\}$ for a 
smooth function $\psi(x,t,\ep)$.  In fact, $\psi(x,t,\ep)$ is the
$r$-coordinate of $\exp_{\Si_t}\big((x,tf(x)),\ep\big)$, i.e. the
$r$-coordinate of the point obtained by following for time $\ep$ the normal
geodesic to $\Si_t$ originating from the point $(x,tf(x))$. In particular,
$\psi(x,0,\ep)=\ep$ and $(\pa\psi/\pa t)(x,0,\ep)=f(x,\ep)$ for a smooth  
function $f(x,\ep)$ satisfying $f(x,0)=f(x)$.  Therefore for $\ep_0>0$
small and fixed and $\ep<<\ep_0$, we have 
\[
\begin{split}
\Big(\int_{r_t>\ep}dv_g\Big)\,\dot{}&=\Big(\int_\Si 
\int_{\psi(x,t,\ep)}^{\ep_0} u_0(x,r)^{-n-1}\sqrt{\frac{\det h_r(x)}{\det h_0(x)}}\,dr
  dv_{h_0}(x)\Big)\,\dot{}  \\
&=- \int_\Si f(x,\ep)u_0(x,\ep)^{-n-1}\sqrt{\frac{\det h_{\ep}(x)}{\det h_0(x)}}\,
dv_{h_0}(x).  
\end{split}
\]
Now $f(x,\ep)\sqrt{\frac{\det h_{\ep}(x)}{\det h_0(x)}}$ is smooth in $\ep$
and equals $f(x)$ at $\ep=0$.  So recalling \eqref{uexpand}, it follows
that the $\log\frac{1}{\ep}$ coefficient in  
$\Big(\int_{r_t>\ep}dv_g\Big)\,\dot{}$ is 
\begin{equation}\label{firstterm}
-(n+1)\int_\Si f\cL \,dh_{0}.   
\end{equation}

To analyze the second term, set $\om_t=-\log (u_t/u_0)$ so that
$u_t=e^{-\om_t}u_0$ and $g_t=e^{2\om_t}g$.  The scalar curvature of $g_t$
is given by
$$
R_{e^{2\om_t}g}=e^{-2\om_t}\left[R_g-2n\Delta_g\om_t
  -n(n-1)|d\om_t|_{g}^2\right].  
$$
Differentiating gives 
\begin{equation}\label{Rdot}
(R_{g_t})\,\dot{}=-2(n\Delta_g\dot{\om} +R_g\dot{\om}).
\end{equation}
Differentiation of \eqref{almostconst} for $g_t$ shows that 
$(R_{g_t})\,\dot{}=O(r^{n+1}\log r)$.  Also
\[
\begin{split}
R_g\dot{\om}&=-n(n+1)\dot{\om} + O(\dot{\om}r^{n+2}\log r) \\
&=-n(n+1)\dot{\om} + O(\frac{\dot{u}r^{n+2}\log r}{u})\\
&=-n(n+1)\dot{\om} + O(r^{n+1}\log r).
\end{split}
\]
Therefore \eqref{Rdot} gives
$$
(n+1)\dot{\om}=\Delta_g \dot{\om} +O(r^{n+1}\log r).  
$$
Hence
\begin{equation}\label{secondterm}
\begin{split}
-(n+1)\int_{r>\ep}\dot{u}u^{-1}\,dv_{g}&=(n+1)\int_{r>\ep}\dot{\om}\,dv_{g}\\
&=\int_{r>\ep}\Delta_g\dot{\om}\,dv_{g} +O(1)\\
&=\int_{r=\ep}\pa_{\nu_g}\dot{\om}\, d\sigma_g +O(1),
\end{split}
\end{equation}
where $\nu_g$ is the outward-pointing $g$-unit normal and $d\sigma_g$ 
the induced area element on $\{r=\ep\}$.  Since $g=u^{-2}(dr^2+h_r)$, 
we have $\pa_{\nu_g}=-u\pa_r$ and $d\sigma_g = u^{-n}dv_{h_\ep}$.  Hence
\begin{equation}\label{newsecondterm}
-(n+1)\int_{r>\ep}\dot{u}u^{-1}\,dv_{g}
=\int_{r=\ep}\big(u^{-n}\dot{u}_r-u^{-n-1}u_r\dot{u}\big)\,dv_{h_\ep} +O(1).
\end{equation}
Now $\dot{r}$ is a smooth function of $(x,r)$ which equals $-f(x)$ at
$r=0$.  So it follows by differentiation of \eqref{utexpand} that $\dot{u}$
takes the form
$$
\dot{u}=-f(x,r) -(n+2)f(x)\cL(x)r^{n+1}\log r +O(r^{n+1}),
$$
where $f(x,r)$ is a smooth function satisfying $f(x,0)=f(x)$.  (This
$f(x,r)$ need not be the same function as the $f(x,\ep)$ which entered into 
the analysis of the first term above.)  Differentiating with respect to
$r$, we conclude that 
$$
\dot{u}_r = s(x,r) -(n+2)(n+1)f(x)\cL(x)r^n\log r +O(r^{n})
$$
for a smooth function $s(x,r)$.  {From} \eqref{uexpand}, it follows that 
\[
\begin{split}
u^{-n}&=r^{-n}\Big(\lambda(x,r) +O(r^{n+1}\log r)\Big)\\
u^{-n-1}&=r^{-n-1}\Big(\mu(x,r) -(n+1)\cL(x)r^{n+1}\log r
+O(r^{n+1})\Big),  
\end{split}
\]
where $\lambda$ and $\mu$ are smooth functions satisfying
$\lambda(x,0)=\mu(x,0)=1$.  Hence the coefficient of 
$\log \frac{1}{\ep}$ in  $\int_{r=\ep}u^{-n}\dot{u}_r\,dv_{h_\ep}$ is
$$
(n+2)(n+1)\int_\Si f\cL\,dv_{h_0}
$$
and the coefficient of $\log \frac{1}{\ep}$ in
$\int_{r=\ep}u^{-n-1}u_r\dot{u}\,dv_{h_\ep}$ is
$$
(n+3)\int_\Si f\cL\,dv_{h_0}.
$$
Combining these in \eqref{newsecondterm}, it follows that the coefficient
of $\log\frac{1}{\ep}$ in $-(n+1)\int_{r>\ep}\dot{u}u^{-1}\,dv_{g}$ is  
$$
(n^2+2n-1)\int_\Si f\cL\,dv_{h_0}.
$$
Combining with \eqref{firstterm} in \eqref{sum} concludes the proof. 
\end{proof}

\section{Calculations}\label{calc}

Recall from \S\ref{asymptotics} that $g=u^{-2}\gb$ and we write
$u=r+r^2\fe$.  The Taylor expansion of $\fe$ is determined by successive
differentiation of \eqref{fe}, and the coefficients $v^{(k)}$ are
determined by the expansion \eqref{expand}.  In this section we outline the
calculation of $v^{(1)}$ and $v^{(2)}$ via this prescription.  In  
particular, this identifies $\cE$ for $n=2$.  We also calculate 
the anomaly in the renormalized volume $V$ for $n=2$.  Throughout this
section we write $\Rb_{\al\be\ga\de}$ and $\Rb$ for the curvature tensor
and scalar curvature of $\gb$, and $R_{ijkl}$ and $R$ for the curvature
tensor and scalar curvature of the induced metric $h=h_0$ on $\Sigma$.  

Equation \eqref{fe0} identifies $\fe|_{r=0}$.  Differentiating \eqref{fe} at
$r=0$ and substituting \eqref{derivs} gives  
\begin{equation}\label{federiv}
3(n-1)\fe_r|_{r=0}=
\frac{1}{n}H^2 -|L|^2-h^{ij}\Rb_{0i0j}+\frac{1}{2n}\Rb.
\end{equation}
Now 
\begin{equation}\label{Rb}
\Rb=\gb^{\al\be}\gb^{\ga\de}\Rb_{\al\ga\be\de}=2h^{ij}\Rb_{0i0j}+h^{ij}h^{kl}\Rb_{ikjl}.   
\end{equation}
The Gauss equation states
$
\Rb_{ikjl}=R_{ikjl}+L_{il}L_{jk}-L_{ij}L_{kl}, 
$
so 
$
h^{ij}h^{kl}\Rb_{ikjl}=R +|L|^2-H^2. 
$
Substituting this in \eqref{Rb} and solving for $h^{ij}\Rb_{0i0j}$ 
gives
\begin{equation}\label{trRb}
h^{ij}\Rb_{0i0j}= \tfrac12\big(\Rb-R-|L|^2+H^2\big).
\end{equation}
Substituting \eqref{trRb} in \eqref{federiv} and then decomposing  
$L_{ij}={\mathring L}_{ij}   
+\frac{1}{n}Hh_{ij}$ gives finally
$$
3(n-1)\fe_r|_{r=0}=\frac{1-n}{2n}\big(\Rb +H^2\big)
+\frac12 \big(R-|\mathring L|^2\big).  
$$
When $n=1$, we have $R=0$ and $\mathring L =0$.  So in this case this   
equation states $0=0$, which shows that $\cL=0$ for $n=1$.  

To calculate $v^{(1)}$ and $v^{(2)}$, first observe that for any
1-parameter family of metrics $h_r$,
$$
\sqrt{\frac{\det h_r}{\det h_0}}=1+\tfrac12(\tr_h h') r 
+\tfrac14\big[\tr_hh''-|h'|_h^2+\tfrac12 (\tr_hh')^2\big]r^2+\cdots.
$$
Substituting \eqref{derivs} and then \eqref{trRb} shows that this becomes 
\begin{equation}\label{det}
\sqrt{\frac{\det h_r}{\det h_0}}=1-Hr
+\tfrac14\big[R-\Rb-|L|^2+H^2\big]r^2 +\cdots.  
\end{equation}
The Taylor expansion of $\fe$ is determined $\mod O(r^2)$ by \eqref{fe0} 
and \eqref{federiv}.  Using this to calculate the expansion of
$(1+r\fe)^{-n-1}$ and then multiplying by \eqref{det} and simplifying, one
finds 
\[
\begin{split}
(1+r\fe)^{-n-1}&\sqrt{\frac{\det h_r}{\det h_0}}\\
=1+&\frac{1-n}{2n}Hr 
+\Big[\frac{n-5}{12(n-1)}\big(R-|\mathring L|^2\big)
+\frac{n-2}{24n^2}\Big((n-3)H^2-2n\Rb\Big)\Big]r^2 + O(r^3).
\end{split}
\]
Thus
\begin{equation}\label{v12}
\begin{split}
v^{(1)}&=\frac{1-n}{2n}H\\
v^{(2)}&=\frac{n-5}{12(n-1)}\big(R-|\mathring L|^2\big)
+\frac{n-2}{24n^2}\Big((n-3)H^2-2n\Rb\Big).
\end{split}
\end{equation}
Substituting the above expressions for $v^{(1)}$, $v^{(2)}$ into
\eqref{coeffs} gives formulae for 
$c_1$, $c_2$ in \eqref{sypvolexp}.  

One can consider volume expansions $\Vol_g(\{\rho>\ep\})$ for other
defining functions $\rho$ of $\Sigma$.  Changing from $r$ to $\rho$ is 
equivalent to changing the choice of background metric from $\gb$ to 
$\Omega^2 \gb$ with $\Omega = |d\rho|_{\gb}$, since
$|d\rho|_{\Omega^2\gb}=1$ so that $\rho$ is the distance to $\Sigma$ in the
metric $\Omega^2 \gb$.  So Proposition~\ref{syeinv}  
implies that the coefficient of $\log \frac{1}{\ep}$ 
(the energy) is independent of the choice of $\rho$.  If one takes 
$\rho =u$, then the coefficients of all the divergent terms are 
integrals of local invariants of $\gb$, just like for $\rho =r$, since the 
Taylor expansion of $u$ 
is locally determined by $\gb$.  In \cite{GoW4}, closed formulae  
are derived for all the coefficients $c_0,\ldots, c_{n-1}, \cE$ for a
general defining function $\rho$.  In the case $\rho = u$,  
the formulae are made explicit in terms of the geometry of $\gb$
for the coefficients $c_1$, $c_2$, $c_3$.  

Observe from \eqref{v12} that $v^{(1)}=0$ when $n=1$.  So also $\cE=0$ when
$n=1$.  When $n=2$ we have  
$$
v^{(2)}=\frac14\big(|\mathring L|^2-R\big).
$$
So the singular Yamabe energy for $n=2$ is
$$
\cE=\frac14\int_\Sigma (|\mathring L|^2-R)\,dv_{h_0}.
$$
One recognizes this as a linear combination of the Willmore energy and the 
Euler characteristic of $\Sigma$.

Finally we derive the anomaly for the renormalized volume for $n=2$.  Let  
$\gb$ and $\gbh=e^{2\om}\gb$ be conformally related metrics, and let $V$
and $\Vh$ be the associated renormalized volumes for $(M,g)$.  The
difference $V-\Vh$ is the constant term in the expansion of 
$\mbox{Vol}_g(\{r>\ep\})-\mbox{Vol}_g(\{\rh>\ep\})$.  Equation \eqref{diff}
gives a formula for this in terms of the function $b(x,\ep)$. 
We calculate enough of the Taylor expansion of $b(x,\ep)$ to evaluate 
the constant term in expansion of the last line of \eqref{diff} when $n=2$.    

The distance $\rh$ for the metric $\gbh$ is determined by the
eikonal equation $|d\rh|^2_{\gbh}=1$.  Using $\gbh=e^{2\om}\gb$ and
writing $\rh=e^\Up r$, this can be written 
$e^{2(\Up -\om)}|dr+rd\Up|^2_{\gb}=1$, or 
$$
2r\Up_r+r^2\big[(\Up_r)^2+h_r^{ij}\Up_i\Up_j\big]=e^{2(\om-\Up)}-1.  
$$
Setting $r=0$ gives $\Up(x,0)=\om(x,0)$.  Differentiating with respect to 
$r$ gives $\Up_r=\frac12 \om_r$ at $r=0$.  Differentiating again gives
$\Up_{rr}=\frac13\big[\om_{rr}+\frac14 (\om_r)^2-\om_i\om^i\big]$ at
$r=0$.  Thus
$$
\Up(x,r)=\om(x,0)+\tfrac12 \om_r(x,0) r 
+\tfrac16 \big[\om_{rr}+\tfrac14 (\om_r)^2-\om_i\om^i\big]r^2 +O(r^3).
$$
Now solve the equation $\rh=e^{\Up(x,r)}r$ for $r$ as a function of $\rh$: 
$r=\rh b(x,\rh)$.  It is elementary to carry this out to obtain
\begin{equation}\label{bexpand}
b(x,\rh)=e^{-\om}
\big[1-\tfrac12 \om_r e^{-\om}\rh 
+\big(\tfrac13(\om_r)^2+\tfrac16\om_i\om^i-\tfrac16\om_{rr}\big)e^{-2\om}\rh^2\big] 
+O(\rh^3),
\end{equation}
where $\om$ and its derivatives are all evaluated at $(x,0)$.  

As noted above, $V-\Vh$ is the constant term in the
expansion of the last line of \eqref{diff}.  Taking $n=2$, this is
the constant term in the expansion of
\begin{equation}\label{diff2}
\int_\Sigma \Big[ -\tfrac12 \ep^{-2}b(x,\ep)^{-2}
-\ep^{-1}v^{(1)}(x)b(x,\ep)^{-1} +v^{(2)}(x)\log b(x,\ep)\Big]\,dv_h. 
\end{equation}
The constant term in $\log b(x,\ep)$ is clearly $-\om(x,0)$.  Easy
calculations manipulating the expansion \eqref{bexpand} show that the
coefficient of 
$\ep$ in the expansion of $b(x,\ep)^{-1}$ is $\frac12 \om_r(x,0)$ and the
coefficient of $\ep^2$ in the expansion of $b(x,\ep)^{-2}$ is 
$\big[\frac{1}{12}(\om_r)^2-\frac13\om_i\om^i+\frac13\om_{rr}\big](x,0)$.
Putting this into \eqref{diff2} along with \eqref{v12} for $n=2$ and 
collecting the terms gives the following.  
\begin{proposition}\label{anomaly}
When $n=2$, the anomaly is given  by
$$
V-\Vh=-\frac18 
\int_\Sigma\Big[2(|\mathring L|^2-R)\om
-H\om_r+\frac{1}{3}\big(4\om_{rr}-4\om_i\om^i+(\om_r)^2)\Big]\,dv_h.
$$
\end{proposition}

It is interesting to compare the singular Yamabe energy and anomaly with
the corresponding quantities $\cE_{\text{min area}}$ and
$(V-\Vh)_{\text{min area}}$ arising from the 
renormalization of the area of the minimal submanifold of the
Poincar\'e-Einstein space associated 
to $\gb$ whose boundary at infinity is equal to $\Sigma$.  This energy was  
calculated in \cite{GrW} to be 
$$
\cE_{\text{min area}} = -\frac18 \int_\Sigma \big(H^2 +
4h^{ij}\Pb_{ij}\big) dv_h,
$$
and the corresponding anomaly was calculated in Proposition 2.2 of
\cite{GrW} to be 
$$
(V-\Vh)_{\text{min area}} = \frac18 
\int_\Sigma\Big[(H^2+4h^{ij}\Pb_{ij})\om 
-2H\om_r+2\om_i\om^i\Big]\,dv_h.
$$
Here
$\Pb_{\al\be}=\frac{1}{n-1}\left(\Rb_{\al\be}-\frac{\Rb}{2n}\gb_{\al\be}\right)$   
denotes the Schouten tensor of $\gb$.  These quantities can be compared to 
those for the singular Yamabe problem via the 
identity $H^2+4h^{ij}\Pb_{ij}=2(|\mathring L|^2+R)$ for $n=2$.   
In particular, 
$$
\cE_{\text{min area}}=-\cE -2\pi \chi(\Sigma).
$$
But clearly there is not such a simple relation between the anomalies.

\end{document}